\documentclass[11pt,a4paper,reqno]{amsart}
\usepackage{amsmath,amssymb,amsthm,amsfonts}
\usepackage[export]{adjustbox}
\usepackage{amsbsy}
\usepackage[utf8]{inputenc}
\usepackage[normalem]{ulem}
\usepackage{xspace}
\usepackage{cancel}
\usepackage[greek,english]{babel}
\usepackage{url}
\usepackage{enumerate}
\usepackage{wrapfig}
\usepackage{enumitem}
\usepackage{multicol}
\usepackage{moreenum}
\usepackage{xcolor}
\usepackage{longtable}
\usepackage{array}

\newcounter{dummy}
\makeatletter
\newcommand\myitem[1][]{\item[#1]\refstepcounter{dummy}\def\@currentlabel{#1}}
\makeatother
\usepackage{subcaption}
\usepackage{comment}
\usepackage{tikz}
\usetikzlibrary{backgrounds}
\usepackage[bottom=2.5cm,top=2.5cm, left=2.5cm, right=2.5cm]{geometry}

\usepackage{multirow}
\usepackage{array,multirow}

\definecolor{LinkColor}{rgb}{0,0,0} 
\usepackage[colorlinks=true,linkcolor=LinkColor,citecolor=LinkColor,urlcolor= LinkColor, naturalnames, hyperindex, pdfstartview=FitH, bookmarksnumbered, plainpages]{hyperref}
\usepackage{cleveref}
\usepackage[pagewise]{lineno}
\makeatletter
\newcommand{\slunlhd}{%
	\mathrel{\mathpalette\sl@unlhd\relax}%
}

\newcommand{\sl@unlhd}[2]{%
	\sbox\z@{$#1\lhd$}%
	\sbox\tw@{$#1\leqslant$}%
	\dimen@=\ht\tw@
	\advance\dimen@-\ht\z@
	\ifx#1\displaystyle
	\advance\dimen@ .2pt
	\else
	\ifx#1\textstyle
	\advance\dimen@ .2pt
	\fi
	\fi
	\ooalign{\raisebox{\dimen@}{$\m@th#1\lhd$}\cr$\m@th#1\leqslant$\cr}%
}
\makeatother

\newtheorem{innercustomthm}{Theorem}[section]

\crefname{innercustomthm}{Theorem}{Theorems}

\newtheorem{Theorem}{Theorem}[section]
\newtheorem{Corollary}[Theorem]{Corollary}
\newtheorem{Lemma}[Theorem]{Lemma}
\newtheorem{Proposition}[Theorem]{Proposition}

\theoremstyle{definition}

\newtheorem{Remark}[Theorem]{Remark}
\newtheorem{Example}[Theorem]{Example}

\newcommand{\Irr}{\operatorname{Irr}}
\newcommand{\para}{\par\vspace{.25cm}}

\newcommand{\U}{\mathcal{U}}

\newcommand{\F}{\mathbb{F}}
\newcommand{\Z}{\mathbb{Z}}

\author[J. Garg]{Jyoti Garg}
\address{Shaheed Udham Singh Government College, Sunam, Punjab, India.}
\email{\href{mailto:jyotigarg0811@gmail.com}{jyotigarg0811@gmail.com}}
\author[S. Maheshwary]{Sugandha Maheshwary}
\address{Department of Mathematics, Indian Institute of Technology Roorkee, Roorkee (Uttarakhand)-247667, India.}
\email{\href{mailto:msugandha@ma.iitr.ac.in}{msugandha@ma.iitr.ac.in}}

\author[H. Setia]{Himanshu Setia}
\address{Department of Mathematics, Indian Institute of Technology Roorkee, Roorkee (Uttarakhand)-247667, India.}
\email{\href{mailto:himansetia@gmail.com}{himansetia@gmail.com}}

\thanks{The second author gratefully acknowledges the support by Science  \& Engineering Research Board (SERB),  DST (Department of Science and Technology), India (SRG/2023/000180).}


\date{}

\title{The structure of $\mathbb{Z}_nG$ and its unit group}
\begin{document}
	\maketitle
	\begin{abstract}

 This article determines the structure of the group ring $\Z_nG$, where $G$ is a finite group and $\Z_n$ is the ring of integers modulo $n$, such that $n$ is relatively prime to the order of $G$. The decomposition of $\Z_nG$ is given as a direct sum of matrix rings over Galois rings, thereby extending the structural theory of group rings beyond the classical field setting. We also provide a method to compute a generating set of the unit group $\mathcal{U}(\Z_nG)$, in terms of  elementary matrices, using Shoda pair theory. The results are illustrated with examples.
	\end{abstract}
	
\section{Introduction}
\noindent Let $G$ be a finite group and let $\Z_{n}$ be the ring of integers modulo $n$ with $\gcd(n,|G|)=1$, where $|G|$ denotes the order of $G$. One of the fundamental objectives of group ring theory is to understand the decomposition of group rings and the structure of their group of units. This  understanding not only helps in dealing with the isomorphism problem or investigating automorphism groups, but also has applications in coding theory. This problem has been studied for $\Z_nG$ in (\cite{KS24},~\cite{KKS25}) for the groups with exponent at most $4.$ Let $n=\prod_{i=1}^{k} p_i^{r_i},$ where $p_{i}$$'$s are distinct primes and $k,r_i \geq 1$. The Chinese remainder theorem implies $\Z_n\cong \oplus_{i=1}^{k} \Z_{p_i^{r_i}}.$ By using \cite[Lemma 1]{PS73}, we have that \begin{equation*}\label{equation1}
		\Z_nG\cong \bigoplus_{i=1}^{k}\Z_{p_i^{r_i}}G.
	\end{equation*} 
	From this, it suffices to determine the structure decomposition of $\mathbb{Z}_{p^r}G$, where $p$ is a prime not dividing $|G|$ and $r \geq 1$.
	 In this article, we determine the structure decomposition of $\Z_{p^r}G,$ where $|G|$ is relatively prime to $p,$
	using the structure decomposition of semisimple group algebra $\F_pG.$ During the last few years, a lot of progress has been done in understanding the Wedderburn decomposition of $\F_pG$ (\cite{AP22},\,\cite{BdR07},\,\cite{BGP13},\,\cite{BGP15},\,\cite{GM19},\,\cite{Mar15}). In Section 3, we provide the structure decomposition of $\Z_{p^r}G$ as direct sum of matrix rings over Galois rings. In section 4, we provide a method to compute a generating set for the unit group of $\Z_{p^r}G$ for a large class of groups called strongly monomial groups, using Shoda pair theory. We have illustrated the theory with explicit computation of generating set for group rings $\Z_{2^r}C_5$, where $r \geq 1$ and $C_5$ a cyclic group of order $5$ and $\mathbb{Z}_{8}G$, where $G$ is an extraspecial group of order $27$ and exponent $3$.
\section{Notation and Preliminaries}
\noindent Let $\F_pG$ be the finite semisimple group algebra of a finite group $G$ over the finite field $\mathbb{F}_{p}$, where $p$ is a prime such that $\gcd(p,|G|)=1$. Let the Wedderburn decomposition of $\F_pG$ be $\F_{p}G\cong \oplus_{i=1}^{k}M_{n_i}(D_i),$ where $D_i$$'$s are division algebras over $\F_p$. Since any finite dimensional division algebra over a finite field $\F_p$ is also a field, it follows that
$$\F_{p}G\cong \bigoplus_{i=1}^kM_{n_i}(\F_{p^{m_i}})$$ for some $n_{i},m_{i}\geq 1.$
\para\noindent Let us define a ring homomorphism $$-:\Z_{p^r}[x] \rightarrow \F_p[x]$$
$$a_0+a_1x+\cdots+a_nx^n \mapsto \bar{a_0}+\bar{a_1}x+\cdots+\bar{a_n}x^n,$$
where $a_i\in \Z_{p^r}$ and $\bar{a_i}={a_i} \mod{p}.$ The image of $\langle h(x) \rangle$ is the ideal $\langle\bar{h}(x)\rangle.$ The above ring homomorphism induces a ring homomorphism $$-:\Z_{p^r}[x]/\langle h(x) \rangle \rightarrow \F_p[x]/\langle \bar{h}(x)\rangle $$ $$a_0+a_1x+\cdots+a_nx^n+\langle h(x) \rangle \mapsto \bar{a_0}+\bar{a_1}x+\cdots+\bar{a_n}x^n+\langle \bar{h}(x)\rangle.$$
\para\noindent A polynomial is said to be {\it basic irreducible polynomial} in $\Z_{p^r}[x]$, if its reduction modulo $p$ is irreducible in $\F_p[x].$ For $m\in \mathbb{N}$, there exists a monic basic irreducible polynomial \cite[Theorem 13.9]{Wan03}, say  $h(x)$,  of degree $m$ over $\Z_{p^r}[x]$, and the quotient
$\Z_{p^r}[x]/\langle h(x)\rangle$ is known as the Galois ring \cite{Wan03} of characteristic $p^r$ and cardinality $p^{rm}$. Let $\xi=x+\left\langle h(x)\right\rangle$, so that $h(\xi)=0$ and $$\sum\limits_{i=0}^{m-1}a_ix^i+\left\langle h(x)\right\rangle =\sum\limits_{i=0}^{m-1}a_i\xi^i.$$ Hence, $\Z_{p^r}[x]/\langle h(x)\rangle=\Z_{p^r}[\xi],$ and all elements of $\Z_{p^r}[x]/\langle h(x)\rangle $ can be expressed uniquely in the form $\sum\limits_{i=0}^{m-1}a_i\xi^i.$
\para\noindent Let us recall the concept of Shoda pairs and strong Shoda pairs of a group and the primitive central idempotents associated to them.\\
A pair $(H,K)$ of subgroups of $G$, such that $H/K$ is cyclic, is called a \textit{Shoda pair}, if for every $g \in G$, we have that  $[H,g]\cap H\subseteq K$ implies $g \in H$. \\
For a Shoda pair $(H,K)$ of $G$, set $\varepsilon(H,K):=\begin{cases}\widehat{K}, & {H=K;} \\
	\prod(\widehat{K}-\widehat{L}), & \hbox{otherwise,}
\end{cases}$ where $\widehat{H}:=\frac{1}{|H|}\displaystyle\sum\limits_{h \in H}h$ and $L$ runs over minimal non-trivial normal subgroups of $H$ containing $K$.\\ The sum of all the distinct $G$-conjugates of $\varepsilon(H,K)$ is usually denoted by $e(G,H,K)$.
A Shoda pair $(H,K)$ of $G$ is called a \textit{strong Shoda pair} \cite[Definition 3.1]{OdRS04} of $G$, if the following conditions hold:\begin{itemize}\item [(i)]$K \leq H \unlhd N_{G}(K)$, where $N_{G}(K)$ is the normaliser of $K$ in $G$;\item[(ii)] $H/K$ is cyclic and a maximal abelian subgroup of $N_{G}(K)/K$;\item [(iii)] for every $g \in G\backslash N_{G}(K),~\varepsilon(H,K)g^{-1}\varepsilon(H,K)g=0$.\end{itemize} If $(H,K)$ is a strong Shoda pair of $G$, then $N_G(K)=\operatorname{Cen}_{G}(\varepsilon(H,K))$, where $\operatorname{Cen}_{G}(\varepsilon(H,K))=\{g \in G ~|~g^{-1}\varepsilon(H,K)g=\varepsilon(H,K)\}$, and if $\lambda$ is a linear character on $H$ with kernel $K$, then the corresponding idempotent $e_{\mathbb{Q}}(\lambda^G)= e(G,H,K)$.
 
\para\noindent Given a Shoda pair $(H,K)$ of $G$, suppose $\operatorname{Irr}(H/K)$ denotes the set of irreducible characters on $H/K$ over $\overline{\mathbb{F}}_{q}$, the algebraic closure of $\mathbb{F}_{q}$. If $k=[H:K]$ and $\xi_{k}$ is a primitive $k$th root of unity in $\overline{\mathbb{F}}_{q}$, then there is a natural action of $\operatorname{Gal}(\mathbb{F}_{q}(\xi_{k})/\mathbb{F}_{q})$ on $\operatorname{Irr}(H/K)$ by composition. The orbits of $\operatorname{Irr}(H/K)$ under this action are called $q$-cyclotomic classes. Denote by $\mathcal{C}_{q}(H/K)$ those $q$-cyclotomic classes which contain generators of $\operatorname{Irr}(H/K)$. If $\mathcal{C} \in \mathcal{C}_{q}(H/K)$ and $\chi \in \mathcal{C}$, set: $$\varepsilon_{\mathcal{C}}(H,K):= |H|^{-1}\sum\limits_{h \in H} \operatorname{tr}(\chi(hK))h^{-1}=[H:K]^{-1}\widehat{K}\sum\limits_{X \in H/K}\operatorname{tr}(\chi(X))h_{X}^{-1},$$ where  $\operatorname{tr}=\operatorname{tr}(\mathbb{F}_{q}(\xi_{k})/\mathbb{F}_{q}$) is the trace of the field extension $\mathbb{F}_{q}(\xi_{k})/\mathbb{F}_{q}$ and $h_{X}$ is a representative of $X \in H/K$.
\para\noindent Define $$e_{\mathcal{C}}(G,H,K):={\rm~the~sum~of~all~the~distinct~}G{\rm {\operatorname{-}} conjugates~of~}\varepsilon_{\mathcal{C}}(H,K).$$
It has been  proved in \cite{BdR07} that if $(H, K)$ is a strong Shoda pair of $G$, then $\operatorname{Cen}_G(\varepsilon_{\mathcal{C}}(H, K)) = E_G(H/K)$, where $E_G(H/K)$ denotes the stabiliser of any element of the cyclotomic class $\mathcal{C}_{q}(H/K)$ under the natural action of the $N_G(K)$.

 In the following theorem, using strong Shoda pairs of $G$, a description of the primitive central idempotents of $\F_pG$ and the associated simple components is given. The ongoing notation is used. 
\begin{Theorem}\cite[Theorem 7]{BdR07}
Let $G$ be a finite group and let $\F_p$ be a field of order $p$ such that $\F_pG$ is semisimple. If $(H, K)$ is a strong Shoda pair of $G$ and $\mathcal{C}\in  \mathcal{C}_q(H/K)$, then $e_{\mathcal{C}}(G, H, K)$ is a primitive central idempotent of $\F_pG$ and $\F_pGe_{\mathcal{C}}(G, H, K) \cong M_{[G:H]}(\mathbb{F}_{p^{{o/[E:H]}}} )$, where $E = E_G(H/K)$ and $o$ is the multiplicative order of $p$ modulo $[H : K]$.
\end{Theorem}

For the convenience of reader, we recall the notion of crossed products. \\

\noindent For a ring $R$ with unity, let $\operatorname{Aut}(R)$  and $\mathcal{U}(R)$ denote the group of automorphisms of $R$ and the group of units of $R$ respectively. Let $\sigma: G \rightarrow \operatorname{Aut}(R)$ and $\tau : G \times G \rightarrow \mathcal{U}(R)$ be two maps which satisfy the following relations,
$$\tau_{gh,x}\sigma_{x}(\tau_{g,h})= \tau_{g,hx} \tau_{h,x}$$ and $$\tau_{g,h}\sigma_{g}(\sigma_{h}(r))= \sigma_{gh}(r) \tau_{g,h},$$for all $g,h,x \in G$ and $r \in R$. Here $\sigma_{g}$ is the image of $g$ under the map $\sigma$ and $\tau_{g,h}$ is the image of $(g,h)$ under the map $\tau.$ Let $R*_{\tau}^{\sigma} G$ denote the set of finite formal sums $\left\{ \sum z_{g}a_{g} ~|~ a_{g} \in R, g \in G \right\}$, where $z_{g}$ is a symbol corresponding to $g \in G.$ Equality and addition in $R*_{\tau}^{\sigma} G$ are defined componentwise. For $g,h \in G$ and $r \in R,$ by setting
$$z_g z_h = z_{gh} \tau_{g,h},$$ $$rz_g = z_g \sigma_g(r)$$ and extending this rule distributively, $R*_{\tau}^{\sigma}G$ becomes an associative ring, called the {\it crossed product} of $G$ over $R$ with twisting $\tau$ and action $\sigma$. Note that $R*_{\tau}^{\sigma}G$ is a free $R$-module with $\{z_{g}~|~g \in G\}$ an $R$-basis, which we call a basis of units of $R*_{\tau}^{\sigma} G$ as an $R$-module. For more details on crossed products, see \cite[Section 2.6]{JdR16}.

By a cyclic crossed product, we shall mean the crossed product $R*_{\tau}^{\sigma}G$, where $R$ over $\mathcal{Z}(R*_{\tau}^{\sigma}G)$ is a cyclic Galois ring extension, i.e., $G=\operatorname{Gal}(R/\mathcal{Z}(R*_{\tau}^{\sigma}G))=\langle \phi \rangle$.

\section{A decomposition of $\Z_{p^r}G,\ r \geq 1$}
In this section, we obtain the structure of $\Z_{p^r}G,\ r \geq 1$, using that of $\mathbb{F}_pG.$ We begin by stating the main theorem.
\begin{Theorem}\label{structure} Let $G$ be a finite group and let $\F_p G$ denote the semisimple group algebra of $G$ over $\F_p$, so that $\gcd(p,|G|)=1$. If the Wedderburn decomposition of $\F_p G$ is
	$$\F_p G \cong \bigoplus_{i=1}^k M_{n_i}(\F_{p^{m_i}}),$$ then for $r\geq 1,$ 
	
		$$\Z_{p^r}G \cong \bigoplus_{i=1}^{k} M_{n_i}(\Z_{p^r}[\xi_{m_i}]),$$
		where $\xi_{m_i}=x+\left\langle h_i(x) \right\rangle,$ $h_i(x)$ being the monic basic irreducible polynomial of degree $m_i$ in $\Z_{p^r}[x]$.

\end{Theorem}
\begin{proof}
	Since $\overline{\xi}_{m_i}=x+\langle{\overline{{h}_i(x)}}\rangle$ is a root of the monic irreducible polynomial ${\overline{{h}_i(x)}}$ over $\F_{p}$ and $\F_{p}[x]/\langle {\overline{{h}_i(x)}}\rangle= \F_{p}[\overline{\xi}_{m_i}]\cong \F_{p^{m_i}},$
	we have that
	$\F_{p}G\cong \bigoplus\limits_{i=1}^{k}M_{n_i}(\F_{p}[\overline{\xi}_{m_i}]).$ 
	Let $$\overline{\phi}: \F_{p}G\rightarrow \bigoplus_{i=1}^{k}M_{n_i}(\F_{p}[\overline{\xi}_{m_i}])$$ be a ring isomorphism.
	
	Any arbitrary element of $\Z_{p^r}G$ is of the form $\sum\limits_{g\in G}a_gg$, where $a_g\in \Z_{p^r}$. Write $a_g=\sum\limits_{l=0}^{r-1}c_{lg}p^l$ with $c_{lg}\in \F_{p}$, so that  
	 $\sum\limits_{g\in G}a_gg=\sum\limits_{l=0}^{r-1}\left(\sum\limits_{g\in G}c_{lg}g\right)p^l.$ Therefore, every element of $\Z_{p^r}G$ is of the form $\sum\limits_{l=0}^{r-1}\alpha_lp^l$, where $\alpha_l\in \F_pG$ for all $0\leq l\leq r-1.$ Define the map $\phi : \mathbb{Z}_{p^r}G \rightarrow \bigoplus\limits_{i=1}^{k} M_{n_i}(\Z_{p^r}[\xi_{m_i}])$ given by 
	$$\sum\limits_{l=0}^{r-1}\alpha_{l}p^l \overset{\phi}\mapsto \left(\left[\sum\limits_{l=0}^{r-1}{a_{uv_{(1)}}^{(l)}}p^l\right]_{n_1},\left[\sum\limits_{l=0}^{r-1}a_{uv_{(2)}}^{(l)}p^l\right]_{n_2},\cdots,\left[\sum\limits_{l=0}^{r-1}a_{uv_{(k)}}^{(l)}p^l\right]_{n_k}\right),$$
	where 
	$$\alpha_l\overset{\overline{\phi}}\mapsto \left(\left[\bar{a}_{uv_{(1)}}^{(l)}\right]_{n_1},\left[\bar{a}_{uv_{(2)}}^{(l)}\right]_{n_2},\cdots,\left[\bar{a}_{uv_{(k)}}^{(l)}\right]_{n_k}\right).$$ Note that  $\bar{a}_{uv_{(i)}}^{(l)}=\sum\limits_{j=0}^{m_i-1}{\mathfrak{a}}^{(l)}_{uv_{(i)}}(j)\bar{\xi}_{m_i}^{j}$ for some ${\mathfrak{a}}^{(l)}_{uv_{(i)}}(j)\in \F_{p}$, and  $a_{uv_{(i)}}^{(l)}=\sum\limits_{j=0}^{m_i-1}{\mathfrak{a}}^{(l)}_{uv_{(i)}}(j)\xi_{m_i}^{j}$, where $1\leq i \leq k$.
	From the definitions of $\overline{\phi}$ and $\phi$, and using that $\overline{\phi}$ is a ring isomorphism, it follows that $\phi$ is a ring homomorphism.\\
	
	\noindent $(i)$ \textbf{\underline{$\phi$ is well defined}}:\\
		Let $\alpha$, $\beta$ be two elements in $ \mathbb{Z}_{p^r}G$, so that  $\alpha=\sum\limits_{l=0}^{r-1}\alpha_lp^l$ and $\beta=\sum\limits_{l=0}^{r-1}\beta_lp^l $, for some $\alpha_l, \beta_l\in \F_pG$.  We check that if $\alpha=\beta$, then $\phi(\alpha)=\phi(\beta)$. Consider \begin{equation}\label{eq1}
		\sum\limits_{l=0}^{r-1}\alpha_lp^l=\sum\limits_{l=0}^{r-1}\beta_lp^l.
	\end{equation} Multiply \Cref{eq1} by $p^{r-1}$, we obtain that $\alpha_0p^{r-1}=\beta_0p^{r-1}.$ Let us assume $\alpha_0=\sum\limits_{g\in G}a_gg$ and $\beta_0=\sum\limits_{g\in G}b_gg,$ where $a_g,b_g\in \F_{p}.$ This implies $$\sum\limits_{g\in G}a_gp^{r-1}g=\sum\limits_{g\in G}b_gp^{r-1}g.$$

As the support of two equal group ring elements is equal, we get that $(a_g-b_g)p^{r-1}=0$ for all $g\in G$. Clearly $a_g=0$ if and only if $b_g=0.$ Assume that both $a_g$ and $b_g$ are non-zero. If $a_g-b_g\neq 0,$ then $a_g-b_g$ is a unit as $a_g-b_g\in \mathbb{F}_p$, implying that $p^{r-1}=0,$ which is not true. So, $a_g=b_g$ for all $g\in G$. Thus, we have $\alpha_0=\beta_0$ and \begin{equation}\label{eq2}
	\sum\limits_{l=1}^{r-1}\alpha_lp^l=\sum\limits_{l=1}^{r-1}\beta_lp^l.
\end{equation} Multiplying equation (\ref{eq2}) by $p^{r-2},$ we obtain that $\alpha_1p^{r-1}=\beta_1p^{r-1}.$ Arguing as above, we obtain that $\alpha_1=\beta_1.$ 
Continuing in this manner, we obtain successively $\alpha_l=\beta_l,$ for all $0\leq l\leq r-1.$ This implies that for every $l$, $\overline{\phi}(\alpha_l)=\overline{\phi}(\beta_l)$ and hence $$ \left(\left[\bar{a}_{uv_{(1)}}^{(l)}\right]_{n_1},\left[\bar{a}_{uv_{(2)}}^{(l)}\right]_{n_2},\cdots,\left[\bar{a}_{uv_{(k)}}^{(l)}\right]_{n_k}\right)=\left(\left[\bar{b}_{uv_{(1)}}^{(l)}\right]_{n_1},\left[\bar{b}_{uv_{(2)}}^{(l)}\right]_{n_2},\cdots,\left[\bar{b}_{uv_{(k)}}^{(l)}\right]_{n_k}\right)$$ with $\bar{a}_{uv_{(i)}}^{(l)}=\sum\limits_{j=0}^{m_i-1}\mathfrak{a}^{(l)}_{uv_{(i)}}(j)\bar{\xi}_{m_i}^{j}$ and $\bar{b}_{uv_{(i)}}^{(l)}=\sum\limits_{j=0}^{m_i-1}\mathfrak{b}^{(l)}_{uv_{(i)}}(j)\bar{\xi}_{m_i}^{j}$, where $\mathfrak{a}^{(l)}_{uv_{(i)}}(j),\mathfrak{b}^{(l)}_{uv_{(i)}}(j)\in \F_{p}$ and $1\leq i \leq k$.  Consequently, we get $\bar{a}_{uv_{(i)}}^{(l)}=\bar{b}_{uv_{(i)}}^{(l)}$, which means $$\sum\limits_{j=0}^{m_i-1}\mathfrak{a}^{(l)}_{uv_{(i)}}(j)\bar{\xi}_{m_i}^{j}=\sum\limits_{j=0}^{m_i-1}\mathfrak{b}^{(l)}_{uv_{(i)}}(j)\bar{\xi}_{m_i}^{j}.$$ Furthermore, the linear independence of $\{1,\bar{\xi}_{m_i},\cdots,\bar{\xi}_{m_i}^{m_i-1}\}$ over $\F_{p}$ implies $\mathfrak{a}^{(l)}_{uv_{(i)}}(j)=\mathfrak{b}^{(l)}_{uv_{(i)}}(j)$ for all $1\leq u,v\leq n_i, 0\leq j\leq m_{i}-1$ and $1\leq i\leq k.$ Therefore, we get $a_{uv_{(i)}}^{(l)}=b_{uv_{(i)}}^{(l)},$ for all $1\leq u,v\leq n_i$ and $1\leq i\leq k.$ Hence, $\phi(\alpha)=\phi(\beta)$.\\ 
	
	\noindent $(ii)$ \textbf{\underline{$\phi$ is injective}}:\\ \noindent 
	We now check that $\phi$ is injective. For this, let $\alpha \in \operatorname{ker}\left(\phi\right)$, so that $\alpha=\sum\limits_{l=0}^{r-1}\alpha_{l}p^l\in \Z_{p^r}G$ and 
	$$\phi(\sum\limits_{l=0}^{r-1}\alpha_{l}p^l)=\textbf{0}.$$
	This gives $$\left(\left[\sum\limits_{l=0}^{r-1}a_{uv_{(1)}}^{(l)}p^l\right]_{n_1},\left[\sum\limits_{l=0}^{r-1}a_{uv_{(2)}}^{(l)}p^l\right]_{n_2},\cdots,\left[\sum\limits_{l=0}^{r-1}a_{uv_{(k)}}^{(l)}p^l\right]_{n_k}\right)=([\textbf{0}]_{n_1},\textbf{[0]}_{n_2},\cdots,\textbf{[0]}_{n_k})$$ and hence $\sum\limits\limits_{l=0}^{r-1}a_{uv_{(i)}}^{(l)}p^l=0$ for all $1\leq i\leq k $ and for all $1\leq u,v\leq n_i$. By the definition of $a_{uv_{(i)}}^{(l)}$, we obtain that
	\begin{equation}\label{eq4}
		\sum\limits_{l=0}^{r-1}\sum\limits_{j=0}^{m_i-1}\mathfrak{a}_{uv_{(i)}}^{(l)}(j)\xi_{m_i}^{j}p^l=0.
	\end{equation}
	Let us denote $\sum\limits_{l=0}^{r-1}\mathfrak{a}_{uv_{(i)}}^{(l)}(j)p^l=\mathcal{A}_{uv}^{i}(j).$ Thus, \Cref{eq4} gives that 
	$\sum\limits_{j=0}^{m_i-1}\mathcal{A}_{uv}^{i}(j)\xi_{m_i}^j=0,$
	where $\mathcal{A}_{uv}^{i}(j)\in \mathbb{Z}_{p^r}.$ Going modulo $p$, we get that 
	$$\sum\limits_{j=0}^{m_i-1}\bar{\mathcal{A}}_{uv}^{i}(j)\bar{\xi}_{m_i}^j=0,$$
	where $\bar{\mathcal{A}}_{uv}^{i}(j) (=\mathfrak{a}^{(l)}_{uv_{(0)}}(j)) \in \mathbb{F}_p$. Furthermore, the linear independence of $\{1,\bar{\xi}_{m_i},\cdots,\bar{\xi}_{m_i}^{m_i-1}\}$ over $\F_{p}$ implies that $\bar{\mathcal{A}}_{uv}^{i}(j)=0$ for all $0\leq j\leq m_i-1.$ This concludes that \begin{equation}\label{eq5}
		\mathcal{A}_{uv}^{i}(j)=pA_{uv}^{i}(j)^{(1)},
	\end{equation}  where $A_{uv}^{i}(j)^{(1)} \in \mathbb{Z}_{p^r}.$ It follows that   
	$$p\left(\sum\limits_{j=0}^{m_i-1}A_{uv}^{i}(j)^{(1)}\xi_{m_i}^j\right)=0$$ for all $1\leq i\leq k$ and for all $1\leq u,v\leq n_i$. Then $\sum\limits_{j=0}^{m_i-1}A_{uv}^{i}(j)^{(1)}\xi_{m_i}^j$ is either $0$ or a zero divisor, i.e., $\sum\limits_{j=0}^{m_i-1}A_{uv}^{i}(j)^{(1)}\xi_{m_i}^j\in p\Z_{p^r}[\xi_{m_i}]$  \cite[Theorem 14.8]{Wan03}.
	\noindent Going modulo $p$, we get $\sum\limits_{j=0}^{m_i-1}\bar{A}_{uv}^{i}(j)^{(1)}\bar{\xi}_{m_i}^j=0$. Reasoning as above, we get that $\bar{A}_{uv}^{i}(j)^{(1)}=0$ for all $0\leq j\leq m_i-1.$ Thus, we may write ${A}_{uv}^{i}(j)^{(1)}=pA_{uv}^{i}(j)^{(2)}$ for all $0\leq j\leq m_i-1$ and $1\leq i\leq k.$ Proceeding in the same way, we obtain successively ${A}_{uv}^{i}(j)^{(3)},{A}_{uv}^{i}(j)^{(4)},\cdots,{A}_{uv}^{i}(j)^{(r)}$, all belongs to $\Z_{p^r}[\xi_{m_i}]$, such that ${A}_{uv}^{i}(j)^{(s)}=p{A}_{uv}^{i}(j)^{(s+1)}$, for $1\leq s\leq r-1.$
	Then, substituting the value in \Cref{eq5}, we get that $$\mathcal{A}_{uv}^{i}(j)=p^rA_{uv}^{i}(j)^{(r)}=0.$$
	This gives $\sum\limits_{l=0}^{r-1}\mathfrak{a}_{uv_{(i)}}^{(l)}(j)p^l=0$, which as seen earlier implies that 
	$\mathfrak{a}_{uv_{(i)}}^{(l)}(j)=0$ and consequently $	{\bar{a}}_{uv_{(i)}}^{(l)}=0$ for all $0\leq l\leq r-1$, $1\leq i\leq k$ and $0\leq j\leq m_i-1$.
	
	\noindent Note that $$\bar{\phi}(\alpha_l)=\left(\left[\bar{a}_{uv_{(1)}}^{(l)}\right]_{n_1},\left[\bar{a}_{uv_{(2)}}^{(l)}\right]_{n_2},\cdots,\left[\bar{a}_{uv_{(k)}}^{(l)}\right]_{n_k}\right).$$ This implies that $\bar{\phi}(\alpha_l)=(0,0,\cdots,0)$ and using the fact that $\bar{\phi}$ is a ring isomorphism, we get that $\alpha_l=0$ for all $0\leq l\leq r-1$. This immediately concludes that $$\sum\limits_{l=0}^{r-1}\alpha_lp^l=0.$$\\
	\noindent $(iii)$ \textbf{\underline{$\phi$ is surjective}}:\\ \noindent For $\phi$ to be surjective, it is enough to prove that the dimensions of $\Z_{p^r}G$ and $\oplus_{i=1}^{k} M_{n_i}(\Z_{p^r}[\xi_{m_i}])$ over $\Z_{p^r}$ are equal.\\ \noindent As $\bar{\phi}$ is a ring isomorphism, computing the dimension of $\F_p G$ and $\oplus_{i=1}^{k}M_{n_i}(\F_p [\overline{\xi}_{m_i}])$ over $\F_p$, we get \begin{equation}\label{eq3}
		|G|=\sum\limits_{i=1}^{k}n_{i}^2m_{i}.
	\end{equation}
	Clearly, the dimension of $\Z_{p^r}G$ over $\Z_{p^r}$ is $|G|$. Also, the dimension of $\oplus_{i=1}^{k} M_{n_i}(\Z_{p^r}[\xi_{m_i}])$ over $\Z_{p^r}$ is $\sum\limits_{i=1}^{k}n_{i}^2m_{i}.$ Hence, the dimensions of $\Z_{p^r}G$ and $\oplus_{i=1}^{k} M_{n_i}(\Z_{p^r}[\xi_{m_i}])$ over $\Z_{p^r}$ are equal, by  \Cref{eq3}. 
\end{proof}

\begin{Remark}
	The ring \(M_{n_i}(\mathbb{Z}_{p^r}[\xi_{m_i}])\) is not simple, since 
	\(M_{n_i}(I)\) is an ideal of \(M_{n_i}(\mathbb{Z}_{p^r}[\xi_{m_i}])\) 
	for every ideal \(I \subseteq \mathbb{Z}_{p^r}[\xi_{m_i}]\).
\end{Remark}

\begin{Corollary}\label{pw}
	Let $G$ be a finite abelian group of order $n$ and let $p$ be a prime such that  $p$ does  not divide $n$. Then 
	$$\Z_{p^r}G \cong \bigoplus_{d \mid n} a_d \Z_{p^r}[\xi_{o_d(p)}],$$
	where $o_d(p)$ is order of $p$ modulo $d$, $a_d=\frac{n_d}{o_d(p)}$ and $n_d$ is the number of elements of order $d$ in $G.$
\end{Corollary}
\begin{proof}
	Let $\zeta_d$ denote a primitive $d^{th}$ root of unity, so that the degree of monic irreducible polynomial of $\zeta_d$ over $\F_{p}$ is $o_d(p)$ \cite{Gue68} and $\left[\F_{p}(\zeta_d):\F_{p}\right]=o_d(p).$ Hence, the corollary follows by \cite[Theorem 3.5.4]{MS02} and \Cref{structure}.
\end{proof}

\begin{Example}{\textbf{The structure of $\Z_{36}(C_5 \times C_{5}).$}}
	
	\noindent By Corollary \ref{pw},
	$$\Z_{2^2}(C_{5} \times C_5)\cong \Z_{2^2} \oplus 6 \Z_{2^2}[\xi_{4}^{(2^2)}]$$  and 
	$$\Z_{3^2}(C_{5} \times C_5) \cong \Z_{3^2} \oplus 6 \Z_{3^2}[\xi_{4}^{(3^2)}].$$ Therefore, $$\Z_{36}(C_{5} \times C_5) \cong \Z_{36} \oplus 6 \Z_{4}[\xi_{4}^{(2^2)}] \oplus 6 \Z_{9}[\xi_{4}^{(3^2)}],$$
	where $\xi_{4}^{(2^2)}$ and $\xi_{4}^{(3^2)}$ are roots of monic basic irreducible polynomials of degree $4,$ over $\Z_{2^2}$ and $\Z_{3^2}$, respectively.
\end{Example}

\begin{Example}{\textbf{The structure of $\Z_{p^r}S_n$,\ $p > n$.}}
	From \cite[Proposition 3.5]{AP22}, we have $$\F_{p}S_n \cong \bigoplus_{\chi \in \Irr(G)}  M_{\chi (1)}(\F_{p}),$$ and \Cref{structure} implies
	$$\Z_{p^r}S_n \cong \bigoplus_{\chi \in \Irr(G)}  M_{\chi (1)}(\Z_{p^r}).$$
	Here, $\Irr(G)$ denotes the set of irreducible characters of $G$ and the value $\chi(1)$ can be calculated as the number of standard
	Young tableaux of shape $\lambda$, where $\lambda$ is a partition of $n.$
\end{Example}

\section{The unit group of $\Z_{p^r}G$}
\noindent In this section, we study $\mathcal{U}(\Z_{p^r}G)$, the unit group of $\Z_{p^r}G$ using its structure determined in the last section. \\
For a group $G$ such that $\operatorname{gcd}(p,|G|)=1$, if	$$\Z_{p^r}G \cong \bigoplus_{i=1}^{k} M_{n_i}(\Z_{p^r}[\xi_{m_i}]),$$ then $$\mathcal{U}(\Z_{p^r}G) \cong \prod_{i=1}^{k} GL_{n_i}(\Z_{p^r}[\xi_{m_i}]).$$ We now aim to find the generators of the unit group $\mathcal{U}(\Z_{p^r}G)$. It is easy to see that \linebreak $GL_{n_i}(\Z_{p^r}[\xi_{m_i}])\cong \mathcal{U}(\Z_{p^r}[\xi_{m_i}])\times SL_{n_i}(\Z_{p^r}[\xi_{m_i}])$. For the commutative components, i.e., if $n_i=1$, the unit group  $\mathcal{U}(M_{n_i}(\Z_{p^r}[\xi_{m_i}]))$ is the unit group of Galois ring $ \mathcal{U}(\Z_{p^r}[\xi_{m_i}])$, which is computed in \cite[Theorem 14.11]{Wan03}. For $n_i> 1$, it is known that the special linear group $SL_{n_i}(\Z_{p^r}[\xi_{m_i}])=E_{n_i}(\Z_{p^r}[\xi_{m_i}])$, the set generated by elementary matrices, i.e., the matrices of the form $I+\boldsymbol{\alpha} E_{ij}$, where $\boldsymbol{\alpha} \in \Z_{p^r}[\xi_{m_i}]$, $I$ denotes the identity matrix, and $E_{ij}$ denotes the matrix with the $(i,j)$-th entry $1$ and all other entries zero \cite[Theorem 4.3.9]{HO89}. Hence, $GL_{n_i}(\Z_{p^r}[\xi_{m_i}])$ is generated by set of elementary matrices and the diagonal matrices of suitable orders.

In subsection 4.1, we compute the generators of the unit group of $\Z_{2^r}C_5$, $r\geq 1$.
 In subsection 4.2, we provide a method to compute the set of generators of $SL_{n_i}(\Z_{p^r}[\xi_{m_i}])$ for a class of groups called strongly monomial groups, using Shoda pair theory. We also provide the explicit calculations for the group ring $\Z_8G$, where $G$ is an extraspecial group of order $27$ and exponent $3$.

\para\noindent We begin by proving a small yet useful result which helps to compute the orders of units in $\Z_{p^r}G$, $r\geq 1$. 
\begin{Lemma}\label{order}
	Let $u$ be a unit of order $l$ in $\F_pG$. Then $u$ is a unit of order $lp^{r-k}$ in $\Z_{p^r}G$, where $k$ is the least positive integer such that $u^l-1 \in \langle p^k \rangle$ and  $u^l-1 \not \in \langle p^{k+1} \rangle$.

\end{Lemma}
\begin{proof}
	Clearly, $\Z_{p^r}G/\langle p \rangle \cong \F_pG.$ If $u \in \mathcal{U}(\F_pG)$ has inverse $v,$ then $uv-1 \in \langle p \rangle$. Since $\langle p \rangle$ is a nil ideal, we have that 
	$uv$ is unipotent unit in $\Z_{p^r}G$ and hence $u$ is invertible in $\Z_{p^r}G.$ 
	
	If order of $u \in \mathcal{U}(\F_pG)$  is $l,$ then $u^l-1=0$ in $\F_pG$ which  implies $u^l-1 \in \langle p \rangle \subseteq \Z_{p^r}G.$ Hence, there exists 
	a least positive integer $k$ such that $u^l-1 \in \langle p^k \rangle $ and $u^l-1\not \in \langle p^{k+1} \rangle.$ Therefore, 
	$$u^l = 1+p^kz,\ \text{where}\ z\neq pz_1\ \text{for any}\ z_1 \in \Z_{p^r}G.$$
	As $u^{lp^{r-k}}=1$ and $u^{lp^{r-k-1}} \neq 1$ in $\Z_{p^r}G,$ we have that order of $u^l$ in $\Z_{p^r}G$ is $p^{r-k}.$ 
	 Consequently, order of $u$ in $\Z_{p^r}G$ is $lp^{r-k}.$
\end{proof}
\subsection{$\mathcal{U}(\Z_{p^r}G),~G$ cyclic}
We begin by studying  $\mathcal{U}(\Z_{p^r}C_n)$, where $C_n$ denotes the cyclic group of order $n$. The general structure of this unit group follows from \Cref{pw}. 

It may be noted that for $n=2$, $\U(\Z_{p^r}C_n)$ has been studied in \cite{KS24}. Also, for $n=3$, a full set of generators of $\U(\Z_{p^r}C_n)$ is obtained, when $p=2$ \cite[Theorem 3.4]{KKS25}, and a partial set of generators is computed, when $p$ is an odd prime \cite[Theorem 4.1, Theorem 4.5]{KKS25}. \\

We now provide the explicit set of generators for $\mathcal{U}(\Z_{2^r}C_5)$, $r\geq 1$.

\begin{Proposition} 
		Let $C_5=\langle g \rangle.$ Then the following hold: 
	\begin{itemize}
		\item[(i)] $\U(\F_2C_5)=\langle 1+g+g^2 \rangle (\simeq C_{15}),$ and
		\item[(ii)] For $r>1$, $$\mathcal{U}(\Z_{2^r}C_5)=G_1 \times G_2 \times G_3 \times G_4 \times G_5 \times G_6 \times G_7,$$  
 where $G_1=\langle 5 \rangle ~(\simeq C_{2^{r-2}}),\ G_2= \langle -1 \rangle ~(\simeq C_2),\ G_3 = \langle 1+g+g^2 \rangle ~(\simeq C_{15.2^{r-1}}),\linebreak \
G_4= \langle 1+2g \rangle~ (\simeq C_{2^{r-1}}),\ G_5= \langle 1+2g^2 \rangle ~(\simeq C_{2^{r-1}}),\ G_6=\langle 1+4g \rangle ~ ( \simeq C_{2^{r-2}}),\ \text{and}\linebreak \ 
 G_7= \langle (1-a)\hat{g}-1 \rangle (\simeq C_2)$~ $\mathrm{with}$~ $a=3.5^{-1}\mod{2^r}.$	
		\end{itemize}
\end{Proposition}
\begin{proof} 
		$(i)$  By the structure of $\F_2C_5$, the unit group $\U(\F_2C_5)\cong C_{15}$ and $1+g+g^2$ is a unit of order $15$ in $\F_2C_5$. Hence, $\U(\F_2C_5)=\langle 1+g+g^2\rangle $.  \para

\noindent $(ii)$ For $r \geq 2,$ using Corollary \ref{pw}, we have $\Z_{2^r}C_5 \cong \Z_{2^r} \oplus \Z_{2^r}[\xi_4].$ Hence, $$\mathcal{U}(\Z_{2^r}C_5) \cong  \mathcal{U}(\Z_{2^r}) \times \mathcal{U}(\Z_{2^r}[\xi_4]).$$ 
	Since $\mathcal{U}(\Z_{2^r})=\langle 5 \rangle \times \langle -1 \rangle \cong C_{2^{r-2}}\times C_2,$ clearly, $\mathcal{U}(\Z_{2^r})=  G_1 \times G_2.$ 
	Also, it follows from \linebreak \cite[Theorem 14.11]{Wan03}, that  $$\mathcal{U}(\Z_{2^r}[\xi_4]) \simeq C_{15.2^{r-1}} \times C_{2^{r-1}} \times C_{2^{r-1}} \times C_{2^{r-2}} \times C_2.$$
Observe that 
	\begin{align*}
		(1+g+g^2)^{15}&=((1+g+g^2)^{3})^5 \\
		&=(g^3+2\alpha_1)^{5}   \\
		&=(1+2\alpha_2),\end{align*}where $\alpha_1=2+2g+3g^2+3g^3+3g^4$ and $\alpha_2=5\alpha_1g^2 + 20\alpha_1^2g^4 + 40\alpha_1^3g + 40\alpha_1^4g^3 + 16\alpha_1^5$. It is clear that $(1+g+g^2)^{15}-1 \in \langle 2 \rangle$ and $(1+g+g^2)^{15}-1 \notin \langle 2^2 \rangle$. Hence, using Lemma \ref{order}, the order of $1+g+g^2$ is $15.2^{r-1}$ in $\Z_{2^r}C_5$. Further,  $(1+2g)^{2^{r-1}}=1 \operatorname{mod} 2^r$ and \linebreak $(1+2g)^{2^{r-2}}=1+2^{r-1}(g+g^2)\neq 1\operatorname{mod} 2^r$ and hence, $1+2g$ is an element of order $2^{r-1}$ in $\Z_{2^r}C_5$. Similarly, the orders of $1+2g^2$ and $1+4g$ are $2^{r-1}$ and  $2^{r-2},$ respectively. Clearly, $(1-a)\hat{g}-1$ is an element of order  $2.$ Thus, in order to prove that 	$\mathcal{U}(\Z_{2^r}C_5)= G_1 \times G_2 \times G_3 \times G_4 \times G_5 \times G_6 \times G_7,$ it suffices to show that if $H_i=\prod\limits_{j=3}^{i}G_j,\ 3 \leq i \leq 7 $, then $H_{k} \cap G_{k+1}=\{1\}$, for every $3 \leq k \leq 6.$ 
	 If $H_k \cap G_{k+1}\neq\{1\}$ for some $k,$ then clearly $H_k$ and $G_{k+1}$ have a common element of order $2.$ Therefore,
	it is sufficient to check that $H_k$ and $G_{k+1}$ have distinct elements of order $2$ for every $3 \leq k \leq 6.$ We observe this by checking that $H_7$ has precisely $2^5-1$ elements of order $2.$ If $g_i$ denotes the element of order $2$ in $G_i,\ 3 \leq i \leq 7,$ then  
	$g_3=1+2^{r-1}(g+g^2+g^3+g^4),\linebreak \ g_4= 1+2^{r-1}(g+g^2),\ g_5=1+2^{r-1}(g^2+g^4),\ g_6= 1+2^{r-1}g,\ g_7=(1-a)\hat{g}-1.$ 
	By direct computations, it can be checked that 
	$\langle g_3 \rangle \times \langle g_4 \rangle \times\langle g_5 \rangle \times\langle g_6 \rangle \times\langle g_7 \rangle$ contains $2^5-1$ distinct  elements of order $2.$ This proves (ii).
\end{proof}
\noindent The above approach suggests that structure of the unit group of $\Z_{p^r}G$ is determinable using that of $\mathcal{U}(\F_{p}G)$. For explicit generators of $\mathcal{U}(\Z_{p^r}G)$, the generators of $\mathcal{U}(\F_{p}G)$ may be lifted suitably.

\subsection{$\mathcal{U}(\Z_{p^r}G),~G$ strongly monomial}
In this subsection, we provide a method to compute the generators of $GL_{n_i}(\Z_{p^r}[\xi_{m_i}])$, $n_i>1$ in $\Z_{p^r}G$. The group $SL_{n_i}(\Z_{p^r}[\xi_{m_i}])$, $n_i>1$ is generated by the set $\{I+\boldsymbol{\alpha} E_{ij}~|~\boldsymbol{\alpha} \in \Z_{p^r}[\zeta_{m_i}]\}$, where $E_{ij}$ denotes the matrix with the $(i,j)$-th entry 1 and all other entries zero \cite[Theorem 4.3.9]{HO89}. As $\mathcal{U}(\Z_{p^r}G) \cong \prod_{i=1}^{k} GL_{n_i}(\Z_{p^r}[\xi_{m_i}]),$ we need to understand this isomorphism.

 One of the methods to study the units of $\mathbb{F}_{p}G$ is by using its primitive central idempotents. As observed in the Section 3, the structure of the group ring $\mathbb{Z}_{p^r}G$ is quite analogous to that of semisimple group algebra $\mathbb{F}_{p}G$. We next explore this possibility of studying the unit group of $\mathbb{Z}_{p^r}G$ by understanding primitive central idempotents of this ring. For this, we first compute the primitive central idempotents of $\Z_{p^r}G$ using those of $\F_pG$ and then relate them to units of $\Z_{p^r}G$.  
	
	In particular, for a strongly monomial group $G$, one of the most effective theories, in the case of semisimple group algebras, is Shoda pair theory and the complete set of primitive central idempotents of $\mathbb{F}_{p}G$ have been studied \cite{BdR07}. We extend the same to obtain substantial results on primitive central idempotents of $\mathbb{Z}_{p^r}G$ in this section and use these results to obtain information on its unit group. 
\para\noindent Let $(H,K)$ be a strong Shoda pair of $G$ and let $\mathcal{C}\in \mathcal{C}_{p}(H/K)$. It follows from  \cite{BdR07} that $\varepsilon_{\mathcal{C}}(H,K)$ is a primitive central idempotent of $\F_pH$. Define $$\omega_{\mathcal{C}}(H,K)=\left(\sum\limits_{k=1}^{n}(-1)^{k-1}\binom{n}{k}\varepsilon_{\mathcal{C}}(H,K)^k\right)^n,$$ where $n$ is the nilpotency index of $\varepsilon_{\mathcal{C}}(H,K)$ in $\mathbb{Z}_{p^r}H$.

We now head towards obtaining a complete set of primitive orthogonal central idempotents of $\mathbb{Z}_{p^r}H$ by lifting   the primitive orthogonal central idempotents of $F_pH$.

\begin{Lemma}\label{lemma_lifting_base_idempotent}
	Let $G$ be a group and let $p$ be a prime not dividing $|G|$. For a Shoda pair $(H,K)$ of $G$, if  $\{\varepsilon_{\mathcal{C}}(H,K)~|~\mathcal{C} \in \mathcal{C}_{p}(H/K)\}$ is a set of primitive central  orthogonal idempotents of $F_pH$, then $\{\omega_{\mathcal{C}}(H,K)~|~\mathcal{C} \in \mathcal{C}_{p}(H/K)\}$ is a set of primitive central orthogonal idempotents of $\mathbb{Z}_{p^r}H$, where $\overline{\omega_{\mathcal{C}}(H,K)}=\varepsilon_{\mathcal{C}}(H,K)$. 
\end{Lemma}
\begin{proof}
	Let $G$ be a group with Shoda pair $(H,K)$. It is easy to observe that the nil ideal of $\mathbb{Z}_{p^r}H$ is $p\mathbb{Z}_{p^r}H$. Therefore, by \cite[Proposition 27.1]{AF92}, the idempotents of $\mathbb{F}_pH$, namely $\varepsilon_{\mathcal{C}}(H,K)$, where $\mathcal{C} \in \mathcal{C}_{p}(H/K)$ can be lifted to $\omega_{\mathcal{C}}(H,K)$ in $\mathbb{Z}_{p^r}H$. As $\varepsilon_{\mathcal{C}}(H,K)$ is central, we readily obtain that  $\omega_{\mathcal{C}}(H,K)$ is a central idempotent. Rather, using \cite[Proposition 27.4]{AF92}, it can be verified that $\{\omega_{\mathcal{C}}(H,K)~|~\mathcal{C} \in \mathcal{C}_{p}(H/K)\}$ is a set of orthogonal idempotents of $\mathbb{Z}_{p^r}H$. Further, each $\omega_{\mathcal{C}}(H,K)$ is primitive. For, if $\omega_{\mathcal{C}}(H,K)=e_1+e_2$, where $e_1,e_2$ are primitive central idempotents of $\mathbb{Z}_{p^r}H$ such that $e_1e_2=0$, then reducing it modulo $p$ gives either $e_1$ or $e_2$ is nilpotent element of $\mathbb{Z}_{p^r}H$, which is not true. Hence, $\{\omega_{\mathcal{C}}(H,K)~|~\mathcal{C} \in \mathcal{C}_{p}(H/K)\}$ is a set of primitive central orthogonal idempotents of $\mathbb{Z}_{p^r}H$.
\end{proof}
 For a strong Shoda pair $(H,K)$ of $G$, and for $\mathcal{C}\in \mathcal{C}_{q}(H/K)$, set $$w_{\mathcal{C}}(G,H,K)  ={\rm~the~sum~of~all~the~distinct~}G{\rm {\operatorname{-}} conjugates~of~}\omega_{\mathcal{C}}(H,K).$$ 
\begin{Proposition}\label{prop_lifting_induced_idempotents} Let $G$ be a group and let $(H,K)$ be a strong Shoda pair of $G$. Then  the following statements hold:
	\begin{enumerate}[label=(\roman*)]
		\item  $\operatorname{Cen}_{G}(\omega_{\mathcal{C}}(H,K))=\operatorname{Cen}_{G}(\varepsilon_{\mathcal{C}}(H,K))$ and $\overline{w_{\mathcal{C}}(G,H,K)}=e_{\mathcal{C}}(G,H,K).$
		\item $w_{\mathcal{C}}(G,H,K)$, $ \mathcal{C} \in \mathcal{C}_{p}(H/K)$, are distinct primitive central orthogonal idempotents of $\mathbb{Z}_{p^r}G$.
	\end{enumerate}
\end{Proposition}
\begin{proof} $(i)$ It follows from the definition of $\omega_{\mathcal{C}}(H,K)$ that $\operatorname{Cen}_{G}(\omega_{\mathcal{C}}(H,K))=\operatorname{Cen}_{G}(\varepsilon_{\mathcal{C}}(H,K))$. Let $T$ be a transversal of $\operatorname{Cen}_{G}(\varepsilon_{\mathcal{C}}(H,K))$ in $G$ and let $n$ be the nilpotency index of $t^{-1}\varepsilon_{\mathcal{C}}(H,K)t$ in $\mathbb{Z}_{p^r}H$, which is same for every $t\in T$. Then considering the orthogonality conditions and direct computations, we get $(i)$. More explicitly,   
	\begin{align*}
		w_{\mathcal{C}}(G,H,K) &=\sum\limits_{t\in T}t^{-1}\omega_{\mathcal{C}}(H,K)t\\
		&=\sum\limits_{t\in T}t^{-1}\left(\sum\limits_{j=1}^{n}(-1)^{j-1}\binom{n}{j}\varepsilon_{\mathcal{C}}(H,K)^j\right)^{n}t \\
		&=\left(\sum\limits_{t\in T}\sum\limits_{j=1}^{n}(-1)^{j-1}\binom{n}{j}(t^{-1}\varepsilon_{\mathcal{C}}(H,K)t)^j\right)^{n}\\
		&=\left(\sum\limits_{j=1}^{n}(-1)^{j-1}\binom{n}{j}\sum\limits_{t\in T}(t^{-1}\varepsilon_{\mathcal{C}}(H,K)t)^j\right)^{n}\\
		&=\left(\sum\limits_{j=1}^{n}(-1)^{j-1}\binom{n}{j}(\sum\limits_{t\in T}t^{-1}\varepsilon_{\mathcal{C}}(H,K)t)^j\right)^{n}\\
		&=\left(\sum\limits_{j=1}^{n}(-1)^{j-1}\binom{n}{j}e_{\mathcal{C}}(G,H,K)^j\right)^{n}.
	\end{align*} 
\noindent Consequently, $\overline{w_{\mathcal{C}}(G,H,K)}=e_{\mathcal{C}}(G,H,K)$. \\
$(ii)$ Further, as $\{t^{-1}\varepsilon_{\mathcal{C}}(H,K)t~|~t\in T\}$ is a set of orthogonal idempotents in $\mathbb{F}_pG$, the lifting of $t^{-1}\varepsilon_{\mathcal{C}}(H,K)t$ is $t^{-1}\omega_{\mathcal{C}}(H,K)t$, which yields that $\{t^{-1}\omega_{\mathcal{C}}(H,K)t~|~t\in T\}$ is an orthogonal set in $\mathbb{Z}_{p^r}G$ \cite[Proposition 27.4]{AF92}. Consequently, $(ii)$ holds by \cite{BdR07}.
\end{proof}
\noindent As a consequence of Proposition \ref{prop_lifting_induced_idempotents} and Theorem 2.1, we observe that if $G$ is a group with a set $S$ of strong Shoda pairs such that $e_{\mathcal{C}}(G, H, K)$, where $(H, K) \in S$ and $\mathcal{C} \in \mathcal{C}_{p}(H/K)$ is a complete set of primitive central idempotents of $\mathbb{F}_{p}G$, then $w_{\mathcal{C}}(G, H, K)$, where $(H, K) \in S$ and $\mathcal{C} \in \mathcal{C}_{p}(H/K)$ is the complete set of primitive central idempotents of $\mathbb{Z}_{p^r} G$. 
\begin{Corollary}\label{SMG}
	If $G$ is a strongly monomial group and $S$ is a complete and irredundant set of strong Shoda pairs of $G$, then $\{w_{\mathcal{C}}(G, H, K)~|~(H, K) \in S ~{\rm and}~ \mathcal{C} \in \mathcal{C}_{p}(H/K)\}$ is the complete set of primitive central orthogonal idempotents of $\mathbb{Z}_{p^r} G$. Consequently, 
	$$\mathbb{Z}_{p^r}Gw_{\mathcal{C}}(G,H,K) \cong  M_{[G:H]}(\mathbb{Z}_{p^r}(\xi_{l_{(H,K)}})), $$
	
	and
	
	$$\Z_{p^r}G \cong \underset{\mathcal{C}\in \mathcal{C}_p(H/K)}{\underset{(H,K)\in S}{\bigoplus}}\mathbb{Z}_{p^r}Gw_{\mathcal{C}}(G,H,K) \cong \bigoplus M_{[G:H]}(\mathbb{Z}_{p^r}(\xi_{l_{(H,K)}})), $$
	where $C = \operatorname{Cen}_{G}(\omega_{\mathcal{C}}(H,K))$, $l_{(H,K)}=\frac{o}{[C:H]}$ and $o$ is the multiplicative order of $p$ modulo $[H:K]$.
\end{Corollary}
\noindent Now, we want to explore the above isomorphism for strongly monomial groups $G$.
\begin{Lemma}\label{Lemma1}
Let $(H,K)$ be a strong Shoda pair of $G$ and let $T=\{t_i~|~1\leq i\leq k\}$ be a right transversal of $C$ in $G$, where $C = \operatorname{Cen}_{G}(\omega_{\mathcal{C}}(H,K))$. Then $\mathbb{Z}_{p^r}Gw_{\mathcal{C}}(G,H,K)\cong M_{k}(\mathbb{Z}_{p^r}C\omega_{\mathcal{C}}(H,K))$ and the map is given by $\alpha \mapsto (\alpha_{ij})_{k\times k}$, where $\alpha_{ij}=\omega_{\mathcal{C}}(H,K)t_j \alpha t_{i}^{-1}\omega_{\mathcal{C}}(H,K)$.
\end{Lemma}
\begin{proof} Define a map $\theta_1:\mathbb{Z}_{p^r}Gw_{\mathcal{C}}(G,H,K)\rightarrow \operatorname{End}_{\mathbb{Z}_{p^r}G}(\mathbb{Z}_{p^r}Gw_{\mathcal{C}}(G,H,K))$ given by $\alpha \mapsto {\theta_{1,\alpha}}$, where ${\theta_{1,\alpha}}(x)=\alpha x$ for $x \in \mathbb{Z}_{p^r}Gw_{\mathcal{C}}(G,H,K)$. This map is a ring isomorphism, since multiplication by elements of the group ring defines $\Z_{p^r}G$-module endomorphisms. Moreover, we have the decomposition $\mathbb{Z}_{p^r}Gw_{\mathcal{C}}(G,H,K) =\bigoplus_{i=1}^{k}\mathbb{Z}_{p^r}G\omega_{\mathcal{C}}(H,K)^{t_i}$. For each $1\leq i\leq k$, the module $\mathbb{Z}_{p^r}G\omega_{\mathcal{C}}(H,K)^{t_k}$ is isomorphic to $ \mathbb{Z}_{p^r}G\omega_{\mathcal{C}}(H,K)$ as $\Z_{p^r}G$-modules. The isomorphism is given by $\gamma\mapsto \gamma t_{k}^{-1}$. Hence, it follows that $\mathbb{Z}_{p^r}Gw_{\mathcal{C}}(G,H,K)$ is isomorphic to direct sum of $k$ copies of $\mathbb{Z}_{p^r}G\omega_{\mathcal{C}}(H,K)$ as $\Z_{p^r}G$ modules. Let $\pi_i:\mathbb{Z}_{p^r}Gw_{\mathcal{C}}(G,H,K)\rightarrow \mathbb{Z}_{p^r}G\omega_{\mathcal{C}}(H,K)$ be the projection onto its $i$-th component and let $\epsilon_j:  \mathbb{Z}_{p^r}G\omega_{\mathcal{C}}(H,K) \rightarrow \mathbb{Z}_{p^r}Gw_{\mathcal{C}}(G,H,K)$ be the inclusion into $j$-th component, where $1\leq i,j\leq k$. The map $\theta_2: \operatorname{End}_{\mathbb{Z}_{p^r}G}(\mathbb{Z}_{p^r}Gw_{\mathcal{C}}(G,H,K)) \rightarrow M_{k}(\operatorname{End}_{\mathbb{Z}_{p^r}G}(\mathbb{Z}_{p^r}G\omega_{\mathcal{C}}(H,K)))$ given by $f\mapsto (f_{ij})_{k}$, where $f_{ij}=(\pi_{i} \circ f\circ \epsilon_{j} ),$ is a ring isomorphism. Finally, the map $\theta_3: M_{k}(\operatorname{End}_{\mathbb{Z}_{p^r}G}(\mathbb{Z}_{p^r}G\omega_{\mathcal{C}}(H,K)))\rightarrow M_{k}(\omega_{\mathcal{C}}(H,K)\mathbb{Z}_{p^r}G\omega_{\mathcal{C}}(H,K))$ given by $M_{k}(f_{ij})\mapsto M_{k}(f_{ij}(\omega_{\mathcal{C}}(H,K)))$ is also a ring isomorphism. Furthermore,\linebreak $\omega_{\mathcal{C}}(H,K)\mathbb{Z}_{p^r}G\omega_{\mathcal{C}}(H,K))=\mathbb{Z}_{p^r}C\omega_{\mathcal{C}}(H,K)$ as the distinct $G$ conjugates of $\omega_{\mathcal{C}}(H,K)$ are mutually orthogonal using Proposition \ref{prop_lifting_induced_idempotents}$(i)$. This proves the lemma.
\end{proof}
\begin{Lemma}
Let $(H,K)$ be a strong Shoda pair of $G$ and $C=\operatorname{Cen}_{G}(\omega_{\mathcal{C}}(H,K))$. Then the $\Z_{p^r}$-module $\Z_{p^r}C\omega_{\mathcal{C}}(H,K)$ is isomorphic to the crossed product $ \mathbb{Z}_{p^r}H\omega_{\mathcal{C}}(H,K)*_{\sigma}^{\tau}C/H$, where $\sigma$ is a faithful action and the twisting $\tau$ is trivial.
\end{Lemma}
\begin{proof}
\noindent The module $\mathbb{Z}_{p^r}C\omega_{\mathcal{C}}(H,K)$ is naturally identified by the crossed product \linebreak  $\mathbb{Z}_{p^r}H\omega_{\mathcal{C}}(H,K)*_{\sigma}^{\tau}C/H$, where the action $\sigma : C/H \rightarrow \operatorname{Aut}(\mathbb{Z}_{p^r}H\omega_{\mathcal{C}}(H,K))$ maps $x$ to the conjugation automorphism $(\sigma)_x$ on $\mathbb{Z}_{p^r}H\omega_{\mathcal{C}}(H,K)$ induced by the fixed inverse image $\bar{x}$ of $x$ under the natural map $C\rightarrow C/H$ and the twisting $\tau: C/H\times C/H \rightarrow \mathcal{U}(\mathbb{Z}_{p^r}H\omega_{\mathcal{C}}(H,K))$ is given by $\tau(x,y)=h\omega_{\mathcal{C}}(H,K)$, where $h\in H$ is such that $\bar{x}\bar{y}=h\overline{xy}$. It is easy to see that $(\sigma)_{x}$ is a faithful action on $\mathbb{Z}_{p^r}H\omega_{\mathcal{C}}(H,K).$ Let $R=\mathbb{Z}_{p^r}H\omega_{\mathcal{C}}(H,K)$ and let $S$ be the center subring of $\mathbb{Z}_{p^r}C\omega_{\mathcal{C}}(H,K)$. Using \Cref{SMG} and \Cref{Lemma1}, $S=\mathbb{Z}_{p^r}(\xi_{l_{(H,K)}})$, where $l_{(H,K)}=\frac{o}{[C:H]}$ and $o$ is the multiplicative order of $p$ modulo $[H:K]$. Then the Galois rings $R$ and $S$ are isomorphic to $GR(p^r,p^{ro})$ and $GR(p^r,p^{rl_{(H,K)}})$ respectively. Moreover, there exists a nonzero element $\zeta$ of order $p^o-1$ such that $R=S[\zeta]$ \cite[Theorem 14.27(i)]{Wan03}. Since the Galois group of $R$ over $S$ is generated by Frobenious automorphism,  $\phi : S\rightarrow S$ given by $\zeta \mapsto \zeta^{p^{l_{(H,K)}}}$, we have that $C/H$ is cyclic. If $C/H=\langle x\rangle$, then a basis of $R*_{\sigma}^{\tau}C/H$ over $R$ is $\{1, \bar{x},\cdots,\bar{x}^{|C/H|-1} \}$, where $\bar{x}$ is the fixed inverse image of $x$ under the natural map $C\rightarrow C/H$. Clearly, $\bar{x}^{|C/H|}\in S$. 
\para \noindent We now check that the twisting $\tau$ is trivial. By \cite[Proposition 2.6.7 (2)]{JdR16}, if $\bar{x}^{|C/H|}=\omega_{\mathcal{C}}(H,K)$, then the twisting is trivial. Otherwise, as the norm map $\mathcal{N}: R\rightarrow S$ is surjective \cite[Theorem 14.37 (ii)]{Wan03}, we can always choose $\alpha \in R$ such that $(\alpha \bar{x})^{|C/H|}= \omega_{\mathcal{C}}(H,K)$. Thus with this new basis $\{1,\alpha\bar{x}, \cdots, (\alpha\bar{x})^{|C/H|-1}\}$ of $R*_{\sigma}^{\tau}C/H$ over $R$, the twisting $\tau$ is trivial.
\end{proof}
\par\noindent  As the twisting is trivial, we denote the crossed product $\mathbb{Z}_{p^r}H\omega_{\mathcal{C}}(H,K)*_{\sigma}^{\tau}C/H$ by $(R/S,\sigma,1)$.
\para\noindent For any prime power $q$ and positive integer $n$, there is a primitive normal basis in $\F_{q^n}$ over $\F_q$ \cite{LS87}. Here by a primitive normal basis of $\F_{q^n}$ over $\F_q$, we mean a normal basis $\{\alpha,\alpha^q,\cdots,\alpha^{q^{n-1}}\}$
such that $\alpha$ also generates the multiplicative group of $\F_{q^n}$. We call such $\alpha$, a primitive normal element of $\F_{q^n}$.
\begin{Lemma}
Let $R$ over $S$ be a finite Galois ring extension with Galois group $\mathcal{G}=\operatorname{Gal}(R/S)$. Then, there exists a normal element $\beta\in R$ such that $\{\theta(\beta)~|~\theta \in \mathcal{G}\}$ forms a basis of $R$, considered as an $S$-module.  
\end{Lemma}
\begin{proof} Clearly, $R/pR$ over $S/pS$ is a finite Galois field extension. If $\mathcal{G}=\langle \theta \rangle$, then the Galois group of $R/pR$ over $S/pS$ is $\mathcal{\bar{G}}=\langle \overline{\theta} \rangle$ \cite[Theorem 14.32]{Wan03}. Hence, there exists a normal element $\alpha$ in $R/pR$ such that $\bar{B}=\{\bar{\theta}(\alpha)~|~\bar{\theta} \in \mathcal{\bar{G}}\}$ is a basis of $R/pR$ over $S/pS$.  Observe that the set $\{\theta(\alpha.1_{R})~|~\theta \in \mathcal{G}\}$, where $1_{R}$ is the identity of $R$, is linearly independent over $S$. Call $\alpha.1_{R}=\beta$, we obtain the required basis for $R$.
\end{proof}
\noindent Analogous to \cite[Proposition 2.6.7 (ii)]{JdR16}, we have the following lemma.
\begin{Lemma}\label{psi}
Let $R$ over $S$ be a finite Galois extension and $n=|\operatorname{Gal}(R/S)|$.
The crossed product $(R/S,\sigma,1)$ is ring isomorphic to $M_{n}(S)$, via isomorphism
$\psi: (R/S,\sigma, 1) \rightarrow M_{n}(S)$ given by
$xu_{\sigma} \mapsto [l_x\circ \sigma]_{B}$,
for $x\in R$ and $\sigma \in \operatorname{Gal}(R/S)$, where $B$ is an $S$-basis of $R$ and $l_x$ denotes multiplication by $x$ on $R$.
\end{Lemma}
\noindent We are now in position to state the next theorem, which provides a complete set of matrix units in each component.
\begin{Theorem}\label{MatrixUnits}
Consider a strong Shoda pair $(H,K)$ of a finite group $G$. Let $R=\Z_{p^r}H\omega_{\mathcal{C}}(H,K)$ and let $S=\mathbb{Z}_{p^r}(\xi_{l_{(H,K)}})$ be the center of $\Z_{p^r}C\omega_{\mathcal{C}}(H,K)$, where $C=\operatorname{Cen}_{G}(\omega_{\mathcal{C}}(H,K))$. If $\beta$ is a normal element of the Galois extension $R$ over $S$, whose Galois group is generated by $\sigma$ and is of order $l$ and $\{1,\gamma,\cdots,\gamma^{l-1}\}$ is a basis of $\Z_{p^r}H\omega_{\mathcal{C}}(H,K)*_{\sigma}^{\tau}C/H$ over $\Z_{p^r}H\omega_{\mathcal{C}}(H,K)$, then there exist $\alpha_{0} , \ldots, \alpha_{l-1} $ in $\Z_{p^r}H\omega_{\mathcal{C}}(H,K)$ satisfying the system
	\begin{equation}\label{eq11}
 	\begin{pmatrix}
 	\beta & \beta^{p^{l_{(H,K)}}} & \cdots & \beta^{p^{(l-1)l_{(H,K)}}}\\
 	\beta^{p^{l_{(H,K)}}} & \beta^{p^{2l_{(H,K)}}} & \cdots & \beta\\
 	\vdots & \vdots & \ddots & \vdots\\
 	\beta^{p^{(l-1)l_{(H,K)}}} & \beta & \cdots & \beta^{p^{(l-2)l_{(H,K)}}}
 	\end{pmatrix}\begin{pmatrix}
 	\alpha_0\\
 	\alpha_1\\
 	\vdots\\
 	\alpha_{l-1}
 	\end{pmatrix}=\begin{pmatrix}
 	 	\sum_{i=1}^{l}  \beta^{p^{il_{(H,K)}}}\\
 	 	\beta- \beta^{p^{l_{(H,K)}}}\\
 	 	\vdots\\
 	 	\beta- \beta^{p^{(l-1)l_{(H,K)}}}
 	 	\end{pmatrix},
 	\end{equation}
 such that $\alpha = \sum_{i=0}^{l-1} \alpha_{i}\gamma^{i}$ is invertible and 
 $$\{ w_{\mathcal{C}}(G,H,K)+\boldsymbol{\alpha}t_{u}^{-1}\gamma^{i} \alpha^{-1} \widehat{E}\, \alpha\, \gamma^{-j}t_{v}\omega_{\mathcal{C}}(H,K)~|~ 1 \leq i,j \leq l,1\leq u,v\leq m,\boldsymbol{\alpha}\in S \}$$ 
	is a complete set of elementary matrices of $\Z_{p^r}Gw_{\mathcal{C}}(G,H,K)$, where $\widehat{E}=\frac{1}{l}\sum_{i=0}^{l-1}\gamma^i$ and \linebreak $\{t_{i}~|~1\leq i\leq m\}$ is a right transversal of $C$ in $G$.
\end{Theorem}
\begin{proof}
Clearly, the existence of $\alpha_{i}$ follows from the observation that the $l\times l$ matrix
$$\begin{pmatrix}
 	\beta & \beta^{p^{l_{(H,K)}}} & \cdots & \beta^{p^{(l-1)l_{(H,K)}}}\\
 	\beta^{p^{l_{(H,K)}}} & \beta^{p^{2l_{(H,K)}}} & \cdots & \beta\\
 	\vdots & \vdots & \ddots & \vdots\\
 	\beta^{p^{(l-1)l_{(H,K)}}} & \beta & \cdots & \beta^{p^{(l-2)l_{(H,K)}}}
 	\end{pmatrix}$$ is invertible.
\para\noindent Since $\alpha_{i}'$s satisfy (\ref{eq11}), we can easily compute the matrix $[\sum_{i=0}^{l-1} l_{\alpha_{i}} \circ \sigma^{i}]_{B}$ and thus \begin{equation}\label{neweq10}
\psi( \sum_{i=0}^{l-1}\alpha_i  \gamma^i)= [\sum_{i=0}^{l-1} l_{\alpha_{i}} \circ \sigma^{i}]_{B}= \begin{pmatrix}
	1 & 1 & 1 & \cdots & 1 & 1\\
	1 & -1 & 0 & \cdots & 0 & 0\\
	1 & 0 & -1 & \cdots & 0 & 0\\
	\vdots & \vdots & \vdots & \ddots & \vdots & \vdots \\
	1 & 0 & 0 & \cdots & -1 & 0\\
	1 & 0 & 0 & \cdots & 0 & -1
\end{pmatrix},
\end{equation}
 which is an invertible matrix. Hence $\alpha= \sum_{i=0}^{l-1} \alpha_{i}\gamma^{i}$ is invertible.
 \para\noindent Further,
 \begin{equation}\label{neweq11}
 	\psi(\widehat{E}) = [\frac{1}{l}\sum_{i=0}^{l-1} \sigma^{i}]_{B}= \frac{1}{l}\begin{pmatrix}
 		1 & 1 & \cdots & 1 & 1\\
 		1 & 1 & \cdots & 1 & 1\\
 		\vdots & \vdots & \ddots & \vdots & \vdots \\
 		1 & 1 & \cdots & 1 & 1\\
 		1 & 1 & \cdots & 1 & 1\\
 	\end{pmatrix}
 \end{equation} and \begin{equation}{\label{eq3}}
  \psi(\alpha^{-1}\widehat{E}\alpha) = \begin{pmatrix}
  1 & 0 & \cdots & 0 & 0\\
  0 & 0 & \cdots & 0 & 0\\
  \vdots & \vdots & \ddots & \vdots & \vdots \\
  0 & 0 & \cdots & 0 & 0\\
  0 & 0 & \cdots & 0 & 0\\
  \end{pmatrix}.
  \end{equation}
\noindent Also, $\psi(\gamma^{i} \alpha^{-1} \widehat{E}\, \alpha\, \gamma^{-j})$ is the $l \times l$ matrix with $(i,j)$-th entry $1$. As $\mathbb{Z}_{p^r}Gw_{\mathcal{C}}(G,H,K)\cong M_{m}(\mathbb{Z}_{p^r}C\omega_{\mathcal{C}}(H,K))$ via the isomorphism given by $t_{u}^{-1}\gamma^{i} \alpha^{-1} \widehat{E}\, \alpha\, \gamma^{-j}t_{v}\mapsto (\alpha_{ij})_{m\times m}$, where \linebreak $\alpha_{ij}=\omega_{\mathcal{C}}(H,K)t_j t_{u}^{-1}\gamma^{i} \alpha^{-1} \widehat{E}\, \alpha\, \gamma^{-j}t_{v} t_{i}^{-1}\omega_{\mathcal{C}}(H,K)$ so that $\alpha_{ij}\neq 0$ only if $i=u,~j=v$ and $\alpha_{uv}=\gamma^{i} \alpha^{-1} \widehat{E}\, \alpha\, \gamma^{-j}\omega_{\mathcal{C}}(H,K)$.
\end{proof}
\subsection{Example}
Consider the group $G$ presented by $$G=\langle a,b,c~|~a^3=b^3=c^3=1,ac=ca, bc=cb, b^{-1}ab=ac^{-1}\rangle.$$ This is a strongly monomial group with a complete set of strong Shoda pairs \cite[Theorem 4]{BM14} given by $$\mathcal{S}_{G}=\{(G,G), (G,\langle a,c \rangle),(\langle a,c\rangle, \langle a\rangle), (G,\langle a^kb,c \rangle)~|~ 0\leq k \leq 2\}.$$ Consider the strong Shoda pair $(H,K)=(\langle a,c\rangle, \langle a\rangle).$ The primitive central idempotent corresponding to this strong Shoda pair in $\mathbb{F}_2G$ is $\varepsilon_{\mathcal{C}}(H,K)=(c+c^2)(1+a+a^2)$, where $\mathcal{C}=[1,2]$. The nilpotency index of $\varepsilon_{\mathcal{C}}(H,K)$ is $3$ in $\Z_{8}G$. The corresponding lifted idempotent in $\mathbb{Z}_{2^3}G$ is $\omega_{\mathcal{C}}(H,K)=(2-c-c^2)(1+a+a^2).$ It follows that $C=\operatorname{Cen}_{G}(\varepsilon_{\mathcal{C}}(H,K))=H$ and $T=\{1,b,b^2\}$ is a transversal of $C$ in $G$. Also, $\mathbb{F}_{2}Ge_{\mathcal{C}}(G,H,K)\cong M_{3}(\mathbb{F}_2H\varepsilon_{\mathcal{C}}(H,K))\cong M_{3}(\mathbb{F}_2({\zeta_3}))$, where $\zeta_3$ is a primitive root of unity. Further, $\mathbb{Z}_{2^3}G\boldsymbol{e}_{\mathcal{C}}(G,H,K)\cong M_{3}(\mathbb{Z}_{2^3}H\omega_{\mathcal{C}}(H,K))$ and  $\mathbb{Z}_{2^3}H\omega_{\mathcal{C}}(H,K))$ is the Galois ring $GR(2^3,2^{3.2})$. Moreover, $GL_{3}(GR(2^3,2^{3.2}))\cong \mathcal{U}(GR(2^3,2^{3.2}))\times SL_{3}(GR(2^3,2^{3.2}))$. Hence, the unit group of $GR(2^3,2^{3.2})$ is isomorphic to $C_3\times C_4\times C_2\times C_2$. The generators of these are $\langle cw_{\mathcal{C}}(G,H,K)\rangle \times \langle (1+2c)w_{\mathcal{C}}(G,H,K)\rangle \times \langle (1+4c)w_{\mathcal{C}}(G,H,K)\rangle \times \langle(1+2c+2c^2)w_{\mathcal{C}}(G,H,K)\rangle.$
\para\noindent Therefore, by \Cref{MatrixUnits}, the set of generators of $SL_{3}(GR(2^3,2^{3.2}))$ is \linebreak$\{w_{\mathcal{C}}(G,H,K)+\boldsymbol{\alpha} b^k\omega_{\mathcal{C}}(H,K)b^{-l}~|~\boldsymbol{\alpha}\in GR(2^3,2^{3.2}) \}$.
\providecommand{\bysame}{\leavevmode\hbox to3em{\hrulefill}\thinspace}
\providecommand{\MR}{\relax\ifhmode\unskip\space\fi MR }
\providecommand{\MRhref}[2]{%
	\href{http://www.ams.org/mathscinet-getitem?mr=#1}{#2}
}
\providecommand{\href}[2]{#2}

\end{document}